\documentclass[12pt,reqno]{amsart}
\usepackage{amsmath,amsthm,amsfonts,amssymb}
\usepackage[foot]{amsaddr}
\usepackage{txfonts}
\usepackage{graphics}
\usepackage{color}
\usepackage{cite}
\usepackage{fleqn}
\usepackage{graphicx}
\usepackage{hyperref}
\usepackage{geometry}
\newtheorem{claim}{Claim}[section]
\newtheorem{theorem}[claim]{Theorem}

\newtheorem{remark}[claim]{Remark}

\usepackage{cite}
\usepackage{hyperref}
\usepackage{color}
\definecolor{Myblue}{cmyk}{1, 1, 0, 0}
\definecolor{Mylily}{cmyk}{0.22, 0.42, 0, 0.40}
\definecolor{Mybrown}{cmyk}{0.21,0.58,1,0.05}
\definecolor{Myred}{cmyk}{0.0,1.0,1.0,0.00}
\definecolor{Mypurple}{rgb}{0.5,0.0,0.5}
\definecolor{Mygreen}{rgb}{0,0.4,0}
\definecolor{grey}{rgb}{.5,.5,.5}

\begin{document}
\begin{center}
{\Large{\textbf{Magnetic Dirichlet Laplacian with radially symmetric magnetic field}}}

\bigskip

{\large{Diana Barseghyan$^{a, b}$, Fran\c{c}oise Truc$^c$}}

\end{center}

\bigskip \begin{quote}

{\small  a) Department of Mathematics, Faculty of Science, University \\ \phantom{c)} of Ostrava, 30.~dubna 22, 70103 Ostrava, Czech Republic \\

b) Department of Theoretical Physics, Nuclear Physics Institute ASCR, \\ \phantom{c)} 25068 \v{R}e\v{z} near Prague, Czech Republic \\[.3em]
  
c) Institut Fourier, UMR 5582 du CNRS Universite de Grenoble I, BP 74, 38402 Saint-Martin d'Heres, France, Bureau 111 \\  

\phantom{c)} \emph{diana.barseghyan@osu.cz, francoise.truc@ujf-grenoble.fr}}

\end{quote}

\bigskip

\noindent \textbf{Abstract.} The aim of the paper is to derive spectral estimates on the eigenvalue moments of the magnetic Dirichlet
 Laplacian defined on the two-dimensional disk with a radially symmetric magnetic field.

\bigskip
\section{Introduction}\setcounter{equation}{0}
Let us consider a particle in a domain $\Omega$ in $\mathbb{R}^2$ in the presence of a magnetic field $B$. We define the 2-dimensional magnetic Laplacian associated 
to this particle as follows:

Let $A$ be a magnetic potential associated to $B$, i.e. a smooth real valued-function on $\Omega \subset \mathbb{R}^2$ verifying   $\mathrm{rot}\, (A)=B$. 
The magnetic Dirichlet Laplacian is initially defined  on $ C_0^\infty (\Omega)$ by
$H_\Omega (A) = (i\nabla + A)^2$.

Under the assumptions that $\Omega$ is a bounded domain and that $A$ satisfies mild regularity conditions,  which means that in particular the magnetic field $B\in L_{\mathrm{loc}}^\infty(\mathbb{R}^2)$ and the corresponding magnetic potential $A\in L^\infty (\Omega)$, the magnetic Sobolev norm $\|(i \nabla+A) u\|_{L^2(\Omega)},\,\,u\in \mathcal{H}_0^1 (\Omega)$, is closed and equivalent to the non-magnetic one, so the self-adjoint Friedrich's extension 
 has a purely discrete spectrum as in the non-magnetic case. 

In the paper we also consider the case when the magnetic field grows to infinity as the variable approaches the boundary 
and has a non zero infimum \begin{equation}\label{grow}B(z)\to\infty\quad\text{as}\quad z\to \partial\Omega \quad\text{and}\quad
K:=\inf\,B(z)>0.\end{equation} In view of the lower bound $$(H_\Omega (A) (u), u)_{L^2(\Omega)}\ge \int_\Omega B(z) |u|^2(z)\,\mathrm{d}z,$$ one  again can construct the Friedrich' s extension of $H_\Omega (A)$ initially defined on $C_0^\infty (\Omega)$. Moreover, it still has a purely discrete spectrum--\cite{T12}.

For simplicity, we will use for the Friedrich's extension the same symbol  $H_\Omega(A)$, and 
we shall denote the increasingly ordered sequence of  its eigenvalues by $\lambda_k =\lambda_k(\Omega, A)$.

The  purpose of  this paper is to establish bounds of the eigenvalue moments of such operators. 
Let us recall the following bound which was proved by Berezin, Li and Yau   for non-magnetic Dirichlet Laplacians on a domain  $\Omega$ in $\mathbb{R}^d$
-- \cite{Be72a, Be72b, LY83},
\begin{equation}
\label{Berezin bound}\sum_k(\Lambda-\lambda_k(\Omega,0))_+^\sigma\le L_{\sigma,d}^{\mathrm{cl}}\,|\Omega|\,
\Lambda^{\sigma+\frac{d}{2}} \quad\text{for any}\;\;\sigma\ge1 \;\;\text{and}\;\;\Lambda>0\,,
\end{equation}
where $|\Omega|$ is the volume of $\Omega$, and the constant on the right-hand side,
\begin{equation}\label{semiclassical constant}
L_{\sigma,d}^{\mathrm{cl}}=\frac{\Gamma(\sigma+1)}{(4\pi)^{\frac{d}{2}}\Gamma
(\sigma+1+d/2)}\,,
\end{equation}
is optimal. Moreover,  for $0\le\sigma<1$, the bound (\ref{Berezin bound})  still exists, but with another constant on the right-hand side
-- \cite{La97} 
\begin{equation}\label{Laptev ineq.}
\sum_k(\Lambda-\lambda_k(\Omega,0))_+^\sigma\le
2\left(\frac{\sigma}{\sigma+1}\right)^\sigma L_{\sigma,d}^{\mathrm{cl}}\,|\Omega|\,
\Lambda^{\sigma+\frac{d}{2}}\,,\quad 0\le\sigma<1\,.
\end{equation}
In the magnetic case, in view of the pointwise diamagnetic inequality which means that under the rather general assumptions on the magnetic potentials \cite{LL01}
 % -------------- %
$$
|\nabla|u(x)||\le|(i\nabla+A)u(x)|\quad\text{for a.a.}\;\; x\in\Omega\,,
$$
we get  that $\lambda_1(\Omega, A)\ge\lambda_1(\Omega,0)$. However, the estimate $\lambda_j(\Omega, A)\ge\lambda_j(\Omega,0)$ fails in general 
if $j\ge2$. Let us mention that,  nevertheless, momentum estimates are still valid for some values of the parameters. In particular, it was shown \cite{LW00} that 
the sharp bound (\ref{Berezin bound}) holds true for arbitrary magnetic fields provided $\sigma\ge\frac{3}{2}$, and  for constant magnetic fields if 
$\sigma\ge1$-- \cite{ELV00}. In the two-dimensional case  the bound (\ref{Laptev ineq.}) holds true for constant magnetic fields if $0\le\sigma<1$, and the constant on 
the  right-hand side cannot be improved --\cite{FLW09}.

In the present work we study the magnetic Dirichlet Laplacian $H_\Omega(A)$ defined on the two- dimensional disk $\Omega$ 
centered in zero and with radius $r_0>0$, with a radially symmetric magnetic field $B(x)=B(|x|)\ge0$.  Our  aim is to extend 
a sufficiently precise  Berezin type inequality to this situation.  A similar problem was studied recently 
in \cite{BEKW16}, but under very strong restrictions on the growth of the magnetic field. 

Let us also mention that some estimates on the counting function of the eigenvalues of the magnetic Dirichlet Laplacian on a disk
 were established in \cite{T12}, in the case where the field is  radial and satisfies some growth condition near the boundary.

\bigskip \section{Main Results}\setcounter{equation}{0}
Before stating the results we define the following one-dimensional operators  $l(B)$ and $\widetilde{l}(B)$  in $L^2((0, r_0), 2\pi r\mathrm{d}r)$ associated with the closures of the quadratic forms
\begin{eqnarray*} Q\left(l(B)\right)[u]= \int_0^{r_0} r \left|\frac{\mathrm{d}u}{\mathrm{d}r}\right|^2\,\mathrm{d}r+\int_0^{r_0}\frac{1}{r}\left(\int_r^1 s B(s)\,
\mathrm{d}s\right)^2|u|^2\,\mathrm{d}r\,,\\Q\left(\widetilde{l}(B)\right)[u]= \int_0^{r_0} r \left|\frac{\mathrm{d}u}{\mathrm{d}r}\right|^2\,
\mathrm{d}r+\int_0^{r_0}\frac{1}{r}\left(\int_0^r s B(s)\,\mathrm{d}s\right)^2|u|^2\,\mathrm{d}r\end{eqnarray*} defined originally on $C_0^\infty(0, r_0)$, 
and acting on their domain as \begin{eqnarray}\nonumber l(B)=-\frac{\mathrm{d}^2}{\mathrm{d}r^2}-\frac{1}{r}\frac{\mathrm{d}}{\mathrm{d}r}+\frac{1}{r^2}
\left(\int_r^{r_0} s B(s)\,\mathrm{d}s\right)^2\,,\\ \label{wide{tilde}{l_0}} \widetilde{l}(B)=-\frac{\mathrm{d}^2}{\mathrm{d}r^2}-\frac{1}{r}\frac{\mathrm{d}}{\mathrm{d}r}
+\frac{1}{r^2}\left(\int_0^r s B(s)\,\mathrm{d}s\right)^2\,.\end{eqnarray}
The following theorem holds true:  \begin{theorem}\label{first theorem}
Let $H_\Omega(A)$ be the magnetic Dirichlet Laplacian on the disk $\Omega$ of radius equal to $r_0$ centered at the origin with a radial magnetic field $B(x)=B(|x|)\ge0$. Let us assume the validity of the mild regularity conditions for the magnetic potential $A$ discussed in Introduction or the validity of (\ref{grow}). Moreover, let
$\int_0^{r_0} s B(s)\,\mathrm{d}s<\infty$.

If $\int_0^{r_0}s B(s)\,\mathrm{d}s\notin\mathbb{Z}$,  then for any $\Lambda\ge0$ and $\sigma\ge3/2$, the following inequality holds 
\begin{eqnarray*}\mathrm{tr}\left(\Lambda-H_\Omega(A)\right)_+^\sigma\le\frac{1}{2}\mathrm{tr}\left(\Lambda-\left(-\Delta_D^\Omega+\frac{1}{x^2+y^2}\left(\int_{\sqrt{x^2+y^2}}^{r_0} s B(s)\,
\mathrm{d}s\right)^2\right)\right)_+^\sigma\\\label{first}\\+\frac{1}{2}\mathrm{tr}\left(\Lambda-\left(-\Delta_D^\Omega+\frac{1}{x^2+y^2}\left(\int_0^{\sqrt{x^2+y^2}} s B(s)\,
\mathrm{d}s\right)^2\right)\right)_+^\sigma\\+\frac{2 L_{\sigma, 1}^{\mathrm{cl}}\,r_0^{2\sigma+1}}{2\sigma+1}\,\left[\int_0^{r_0}s B(s)\,\mathrm{d}s\right]\,
\Lambda^{\sigma+1/2}+\frac{1}{2}\mathrm{tr}\left(\Lambda-l(B)\right)_+^\sigma+\frac{1}{2}\mathrm{tr}\left(\Lambda-
\widetilde{l}(B)\right)_+^\sigma\,\,\end{eqnarray*}
If $\int_0^{r_0}s B(s)\,\mathrm{d}s\in\mathbb{Z}$,  then for any $\Lambda\ge0$ and $\sigma\ge3/2$, the following inequality holds  
\begin{eqnarray*}\mathrm{tr}\left(\Lambda-H_\Omega(A)\right)_+^\sigma\le\frac{1}{2}\mathrm{tr}\left(\Lambda-\left(-\Delta_D^\Omega+\frac{1}{x^2+y^2}\left(\int_{\sqrt{x^2+y^2}}^{r_0} s B(s)\,
\mathrm{d}s\right)^2\right)\right)_+^\sigma\\\label{first}\\+\frac{1}{2}\mathrm{tr}\left(\Lambda-\left(-\Delta_D^\Omega+\frac{1}{x^2+y^2}\left(\int_0^{\sqrt{x^2+y^2}} s B(s)\,
\mathrm{d}s\right)^2\right)\right)_+^\sigma\\+\frac{2 L_{\sigma, 1}^{\mathrm{cl}}\,r_0^{2\sigma+1}}{2\sigma+1}\left[\int_0^{r_0}s B(s)\,\mathrm{d}s\right]\,
\Lambda^{\sigma+1/2}-\frac{1}{2}\mathrm{tr}\left(\Lambda-l(B)\right)_+^\sigma+\frac{1}{2}\mathrm{tr}\left(\Lambda-
\widetilde{l}(B)\right)_+^\sigma\,.\end{eqnarray*}
where the operators $l(B)$ and $\widetilde{l}(B)$ are defined in (\ref{wide{tilde}{l_0}}) and 
$L_{\sigma, 1}^{\mathrm{cl}}$ is the semiclassical constant given by (\ref{semiclassical constant}).\end{theorem}
\begin{proof} We  use the standard partial wave decomposition --\cite{E96} $$L^2(\Omega)=\bigoplus_{m=-\infty}^\infty L^2((0, r_0), 2\pi r\,\mathrm{d}r)$$ and $$H_\Omega(A)=\bigoplus_{m=-\infty}^\infty h_m(B),$$ where the operators $h_m(B)$ in $L^2((0, r_0), 2\pi r\mathrm{d}r)$ are associated with the closures of the quadratic forms
$$Q(h_m(B))[u]= \int_0^{r_0} r \left|\frac{\mathrm{d}u}{\mathrm{d}r}\right|^2\,\mathrm{d}r+\int_0^{r_0}\left(\frac{m}{r}-\frac{1}{r}\int_0^r s B(s)\,\mathrm{d}s\right)^2 r |u|^2\,\mathrm{d}r,$$ defined originally on $C_0^\infty(0, r_0)$, and acting on their domain as
$$h_m(B)=-\frac{\mathrm{d}^2}{\mathrm{d}r^2}-\frac{1}{r}\frac{\mathrm{d}}{\mathrm{d}r}+\left(\frac{m}{r}-\frac{1}{r}\int_0^r s B(s)\,\mathrm{d}s\right)^2.$$
Let us first consider the case where $m\ge\left[\int_0^{r_0} s B(s)\,\mathrm{d}s\right]+1$. Then 
\begin{eqnarray}\nonumber\bigoplus_{m\ge\left[\int_0^{r_0} s B(s)\,\mathrm{d}s\right]+1}h_m(B)=\bigoplus_{k\ge1}\left(-\frac{\mathrm{d}^2}{\mathrm{d}r^2}-\frac{1}{r}\frac{\mathrm{d}}{\mathrm{d}r}+\frac{1}{r^2}\left(\left[\int_0^{r_0} s B(s)\,\mathrm{d}s\right]+k-\int_0^r s B(s)\,\mathrm{d}s\right)^2\right)\\\nonumber=\bigoplus_{k\ge1}\left(-\frac{\mathrm{d}^2}{\mathrm{d}r^2}-\frac{1}{r}\frac{\mathrm{d}}{\mathrm{d}r}+\frac{1}{r^2}\left(\int_0^{r_0} s B(s)\,\mathrm{d}s-\left\{\int_0^{r_0} s B(s)\,\mathrm{d}s\right\}+k-\int_0^r s B(s)\,\mathrm{d}s\right)^2\right)\\\label{first case}=\bigoplus_{k\ge1}\left(-\frac{\mathrm{d}^2}{\mathrm{d}r^2}-\frac{1}{r}\frac{\mathrm{d}}{\mathrm{d}r}+\frac{1}{r^2}\left(\int_r^{r_0} s B(s)\,\mathrm{d}s-\left\{\int_0^{r_0} s B(s)\,\mathrm{d}s\right\}+k\right)^2\right)\end{eqnarray}
If $\int_0^{r_0}s B(s)\,\mathrm{d}s\notin\mathbb{Z}$ then we get from (\ref{first case})
\begin{eqnarray*}\bigoplus_{m\ge\left[\int_0^{r_0} s B(s)\,\mathrm{d}s\right]+1}h_m(B)
\ge\bigoplus_{n\ge0}\left(-\frac{\mathrm{d}^2}{\mathrm{d}r^2}-\frac{1}{r}\frac{\mathrm{d}}{\mathrm{d}r}+\frac{1}{r^2}\left(\int_r^{r_0} s B(s)\,\mathrm{d}s+n\right)^2\right)\\\ge\bigoplus_{n\ge0}\left(-\frac{\mathrm{d}^2}{\mathrm{d}r^2}-\frac{1}{r}\frac{\mathrm{d}}{\mathrm{d}r}+\frac{n^2}{r^2}+\frac{1}{r^2}\left(\int_r^{r_0} s B(s)\,\mathrm{d}s\right)^2\right).\end{eqnarray*}
Therefore for any $\Lambda\ge0$ and $\sigma\ge0$ \begin{equation}\label{Hl}\mathrm{tr}\left(\Lambda-\bigoplus_{m\ge\left[\int_0^{r_0}s B(s)\,\mathrm{d}s\right]+1}h_m(B)\right)_+^\sigma\le\mathrm{tr}\left(\Lambda-\bigoplus_{n=0}^\infty l_n(B)\right)_+^\sigma,\end{equation} where the operators $l_n(B)$ in $L^2((0, r_0), 2\pi r\mathrm{d}r)$ are associated with the closures of the quadratic forms
$$Q\left(l_n(B)\right)[u]= \int_0^{r_0} r \left|\frac{\mathrm{d}u}{\mathrm{d}r}\right|^2\,\mathrm{d}r+\int_0^1\left(\frac{n^2}{r}+\frac{1}{r}\left(\int_r^{r_0} s B(s)\,\mathrm{d}s\right)^2\right)  |u|^2\,\mathrm{d}r,$$ defined originally on $C_0^\infty(0, r_0)$, and acting on their domain as
$$l_n(B)=-\frac{\mathrm{d}^2}{\mathrm{d}r^2}-\frac{1}{r}\frac{\mathrm{d}}{\mathrm{d}r}+\frac{n^2}{r^2}+\frac{1}{r^2}\left(\int_r^{r_0} s B(s)\,\mathrm{d}s\right)^2.$$
On the other hand, if $\int_0^{r_0}s B(s)\,\mathrm{d}s\in\mathbb{Z}$ then (\ref{first case})  writes
\begin{eqnarray*}\bigoplus_{m\ge\left[\int_0^{r_0} s B(s)\,\mathrm{d}s\right]+1}h_m(B)=\bigoplus_{k\ge1}\left(-
\frac{\mathrm{d}^2}{\mathrm{d}r^2}-\frac{1}{r}\frac{\mathrm{d}}{\mathrm{d}r}+\frac{1}{r^2}\left(\int_r^{r_0} s B(s)\,\mathrm{d}s+k\right)^2\right)
\\\ge\bigoplus_{k\ge1}\left(-\frac{\mathrm{d}^2}{\mathrm{d}r^2}-\frac{1}{r}\frac{\mathrm{d}}{\mathrm{d}r}+\frac{k^2}{r^2}+\frac{1}{r^2}
\left(\int_r^{r_0} s B(s)\,\mathrm{d}s\right)^2\right)\end{eqnarray*} and, similarly, 
\begin{equation}\label{hl1}\mathrm{tr}\left(\Lambda-\bigoplus_{m\ge\left[\int_0^{r_0}s B(s)\,\mathrm{d}s\right]+1}h_m(B)\right)_+^\sigma\le\mathrm{tr}\left(\Lambda-\bigoplus_{k=1}^\infty l_k(B)\right)_+^\sigma\,.\end{equation}
Using the following trace symmetry of the operators $l_n(B)$ with respect to $n$
$$\mathrm{tr}\left(\Lambda-\bigoplus_{n>0}l_n(B)\right)_+^\sigma=\mathrm{tr}\left(\Lambda-\bigoplus_{n<0}l_n(B)\right)_+^\sigma$$ 
we write $$\mathrm{tr}\left(\Lambda-\bigoplus_{n\in\mathbb{Z}}l_n(B)\right)_+^\sigma=\mathrm{tr}\left(\Lambda-\bigoplus_{n>0}l_n(B)\right)_+^\sigma+\mathrm{tr}\left(\Lambda-\bigoplus_{n<0}l_n(B)\right)_+^\sigma+\mathrm{tr}\left(\Lambda-l(B)\right)_+^\sigma$$$$=2\mathrm{tr}\left(\Lambda-\bigoplus_{n>0}l_n(B)\right)_+^\sigma+\mathrm{tr}\left(\Lambda-l(B)\right)_+^\sigma,$$ where the operator $l(B)$ is defined in (\ref{wide{tilde}{l_0}}), so that  inequality (\ref{Hl}) implies
\begin{equation}\label{Lambda}\mathrm{tr}\left(\Lambda-\bigoplus_{m\ge\left[\int_0^{r_0}s B(s)\,\mathrm{d}s\right]+1}h_m(B)\right)_+^\sigma\leq\frac{1}{2}\mathrm{tr}
\left(\Lambda-\bigoplus_{n\in\mathbb{Z}} l_n(B)\right)_+^\sigma+\frac{1}{2}\mathrm{tr}\left(\Lambda-l(B)\right)_+^\sigma,\end{equation} 
if $\int_0^{r_0}s B(s)\,\mathrm{d}s\notin\mathbb{Z}$,

\bigskip
and inequality (\ref{hl1}) implies  \begin{equation}\label{lambda}\mathrm{tr}\left(\Lambda-\bigoplus_{m\ge\left[\int_0^{r_0}s B(s)\,\mathrm{d}s\right]+1}h_m(B)\right)_+^\sigma\leq\frac{1}{2}\mathrm{tr}\left(\Lambda-\bigoplus_{n\in\mathbb{Z}} l_n(B)\right)_+^\sigma-\frac{1}{2}\mathrm{tr}\left(\Lambda-l(B)\right)_+^\sigma \end{equation} if $\int_0^{r_0}s B(s)\,\mathrm{d}s\in\mathbb{Z}$. 

\bigskip
Applying the partial wave decomposition for the two-dimensional Schr\"{o}dinger operator $-\Delta_D^\Omega+\frac{1}{x^2+y^2}\left(\int_{\sqrt{x^2+y^2}}^{r_0} s B(s)\,\mathrm{d}s\right)^2$ we obtain from the inequalities (\ref{Lambda}) and (\ref{lambda}) that, for any $\sigma\ge0$\begin{eqnarray}\nonumber\mathrm{tr}\left(\Lambda-\bigoplus_{m\ge\left[\int_0^{r_0}s B(s)\,\mathrm{d}s\right]+1}h_m(B)\right)_+^\sigma\le\frac{1}{2}\mathrm{tr}\left(\Lambda-\left(-\Delta_D^\omega+\frac{1}{x^2+y^2}\left(\int_{\sqrt{x^2+y^2}}^{r_0} s B(s)\,\mathrm{d}s\right)^2\right)\right)_+^\sigma\\\label{first}+\frac{1}{2}\mathrm{tr}\left(\Lambda-l(B)\right)_+^\sigma,\,\end{eqnarray} if $\int_0^{r_0}s B(s)\,\mathrm{d}s\notin\mathbb{Z}$, and that

 \begin{eqnarray}\nonumber\mathrm{tr}\left(\Lambda-\bigoplus_{m\ge\left[\int_0^{r_0}s B(s)\,\mathrm{d}s\right]+1}h_m(B)\right)_+^\sigma\le\frac{1}{2}\mathrm{tr}\left(\Lambda-\left(-\Delta_D^\omega+\frac{1}{x^2+y^2}\left(\int_{\sqrt{x^2+y^2}}^{r_0} s B(s)\,\mathrm{d}s\right)^2\right)\right)_+^\sigma\\\label{first1}-\frac{1}{2}\mathrm{tr}\left(\Lambda-l(B)\right)_+^\sigma,\,\end{eqnarray} if $\int_0^{r_0}s B(s)\,\mathrm{d}s\in\mathbb{Z}$.

\bigskip
Now we move on to the case where $m\le0$. Since $$h_m(B)\ge-\frac{\mathrm{d}^2}{\mathrm{d}r^2}-\frac{1}{r}\frac{\mathrm{d}}{\mathrm{d}r}+\frac{m^2}{r^2}+\frac{1}{r^2}\left(\int_0^r s B(s)\,\mathrm{d}s\right)^2$$ then, 
repeating the same ideas as before we arrive to
\begin{eqnarray}\nonumber \mathrm{tr}\left(\Lambda-\bigoplus_{m\le0}h_m(B)\right)_+^\sigma=\mathrm{tr}\left(\Lambda-\bigoplus_{m<0}h_m(B)\right)_+^\sigma+\mathrm{tr}\left(\Lambda-\widetilde{l}(B)\right)_+^\sigma\\\label{second}\le\frac{1}{2}\mathrm{tr}\left(\Lambda-\left(-\Delta_D^\omega+\frac{1}{x^2+y^2}\left(\int_0^{\sqrt{x^2+y^2}} s B(s)\,\mathrm{d}s\right)^2\right)\right)_+^\sigma+\frac{1}{2}\mathrm{tr}\left(\Lambda-\widetilde{l}(B)\right)_+^\sigma,\end{eqnarray}
where $\widetilde{l}(B)$ is given by (\ref{wide{tilde}{l_0}}).

\bigskip
Finally let $1\le m\le\left[\int_0^{r_0}s B(s)\,\mathrm{d}s\right]$. We follow the method of \cite{T12}. Performing the change of the variables $r=r_0 e^t$ in the quadratic form $Q_m\left(h_m(B)-\Lambda\right)$ corresponding to $h_m(B)-\Lambda$ \begin{equation}\label{qform}Q_m\left(h_m(B)-\Lambda\right)=\int_0^{r_0} r |u^\prime|^2\,\mathrm{d}r+\int_0^{r_0} r\left(\left(\frac{m}{r}-\frac{1}{r}\int_0^r s B(s)\,\mathrm{d}s\right)^2-\Lambda\right) |u|^2\,\mathrm{d}r\end{equation} and defining $w$ by $w(t)=u(r_0 e^t)$ we transfer (\ref{qform}) to \begin{equation}\label{transfer}\widetilde{Q}_m(B)=\int_{-\infty}^0 |w^\prime|^2\,\mathrm{d}t+\int_{-\infty}^0 \left(\left(m-\int_0^{r_0 e^t} s B(s)\,\mathrm{d}s\right)^2-\Lambda r_0^2 e^{2t}\right) |w|^2\,\mathrm{d}t,\end{equation} where $\widetilde{Q}_m(B)$ is defined on $\mathcal{H}_0^1(-\infty, 0)$ and corresponds to the one-dimensional operator $g_{m, B}=-\frac{\mathrm{d}^2}{\mathrm{d}t^2}+\left(m-\int_0^{r_0 e^t} s B(s)\,\mathrm{d}s\right)^2-\Lambda r_0^2 e^{2t}$. 

It is easy to notice that for any positive $\varepsilon<1$ \begin{equation}\label{g} g_{m, B}\ge g_B+(1-\varepsilon)m^2,\end{equation} where the operator $$g_B=-\frac{\mathrm{d}^2}{\mathrm{d}t^2}-\left(\frac{1}{\varepsilon}-1\right)\left(\int_0^{r_0 e^t} s B(s)\,\mathrm{d}s\right)^2-\Lambda r_0^2 e^{2t}$$ is defined on $\mathcal{H}_0^1(-\infty, 0)$.  

Let $\{\mu_k(B)\}_{k=1}^\infty$ be the set of the negative eigenvalues of $g_{m, B}$. Then due to  the minimax principle (\ref{g})
implies \begin{eqnarray}\nonumber\mathrm{tr}\left(\Lambda-\bigoplus_{m=1}^{\left[\int_0^{r_0}s B(s)\,\mathrm{d}s\right]}h_m(B)\right)_+^\sigma=\sum_{m=1}^{\left[\int_0^{r_0}s B(s)\,\mathrm{d}s\right]}\mathrm{tr}\left(\Lambda-h_m(B)\right)_+^\sigma\\\nonumber=\sum_{m=1}^{\left[\int_0^{r_0}s B(s)\,\mathrm{d}s\right]}\mathrm{tr}(g_{m, B})_-^\sigma\le\sum_{m=1}^{\left[\int_0^{r_0}s B(s)\,\mathrm{d}s\right]}\mathrm{tr}(g_B+(1-\varepsilon) m^2))_-^\sigma\\\nonumber\le\sum_{m=1}^{\left[\int_0^{r_0}s B(s)\,\mathrm{d}s\right]}\sum_{\mu_k(B)+(1-\varepsilon) m^2\le0}|\mu_k(B)+(1-\varepsilon) m^2|^\sigma\\\nonumber\le\sum_{|\mu_k(B)|\le(1-\varepsilon)\,\left[\int_0^{r_0}s B(s)\,\mathrm{d}s\right]^2}\,\sum_{1\le |m|\le\frac{1}{\sqrt{1-\varepsilon}}\sqrt{|\mu_k(B)|}}|\mu_k(B)+(1-\varepsilon) m^2|^\sigma\\\nonumber+\sum_{|\mu_k(B)|>(1-\varepsilon)\left[\int_0^{r_0}s B(s)\,\mathrm{d}s\right]^2}\,\sum_{1\le |m|\le\left[\int_0^{r_0}s B(s)\,\mathrm{d}s\right]}|\mu_k(B)+(1-\varepsilon) m^2|^\sigma\\\label{mu}\le2\left[\int_0^{r_0}s B(s)\,\mathrm{d}s\right]\,\sum_k|\mu_k(B)|^\sigma \end{eqnarray}
Let us extend the potential $-\left(\frac{1}{\varepsilon}-1\right)\left(\int_0^{r_0 e^t} s B(s)\,\mathrm{d}s\right)^2-\Lambda r_0^2 e^{2t}$ to $\mathbb{R}$ by zero and denote the corresponding one dimensional Schr\"{o}dinger operator by $\widetilde{g}_B$. (We omit the dependence on $\epsilon$ in the notation for simplicity). Since $C_0^\infty(-\infty, 0)\subset C_0^\infty(\mathbb{R})$ then \begin{equation}\label{nu}\sum_k|\mu_k(B)|^\sigma\le\sum_k|\nu_k(B)|^\sigma,\end{equation} where $\{\nu_k(B)\}_{k=1}^\infty$ are the negative eigenvalues of $\widetilde{g}_B$. 

Applying the trace formulae --\cite{LT76} for any $\sigma\ge3/2$ we get
$$\sum_k|\nu_k(B)|^\sigma\le L_{\sigma, 1}^{\mathrm{cl}}\int_{-\infty}^0\left(\left(\frac{1}{\varepsilon}-1\right)\left(\int_0^{r_0 e^t} s B(s)\,\mathrm{d}s\right)^2+\Lambda r_0^2 e^{2t}\right)^{\sigma+1/2}\,\mathrm{d}t$$$$=L_{\sigma, 1}^{\mathrm{cl}}\int_0^{r_0}\left(\left(\frac{1}{\varepsilon}-1\right)\left(\int_0^r s B(s)\,\mathrm{d}s\right)^2+ \Lambda r^2\right)^{\sigma+1/2}\,\frac{1}{r}\,\mathrm{d}r.$$
After passing to the limit $\varepsilon\to1$ the above inequality yields \begin{equation}\label{number}\sum_k|\nu_k(B)|^\sigma\le\frac{r_0^{2\sigma+1}}{2\sigma+1}\,L_{\sigma, 1}^{\mathrm{cl}} \,\Lambda^{\sigma+1/2}\,.\end{equation}
By virtue of the estimates (\ref{mu})-(\ref{number}) we obtain 
\begin{eqnarray*}\mathrm{tr}\left(\Lambda-\bigoplus_{m=1}^{\left[\int_0^{r_0}s B(s)\,\mathrm{d}s\right]}h_m(B)\right)_+^\sigma\le\frac{2L_{\sigma, 1}^{\mathrm{cl}} r_0^{2\sigma+1}}{2\sigma+1}\left[\int_0^{r_0}s B(s)\,\mathrm{d}s\right]\,\Lambda^{\sigma+1/2}\,,\end{eqnarray*} which together with (\ref{first})- (\ref{second}) establishes the theorems. 

\end{proof}

\begin{remark}
Let us assume that $\int_0^{r_0} s B(s)\,\mathrm{d}s<1$. In view of Theorem\,2.1 the threshold of the spectrum of the corresponding magnetic Laplacian can be estimated from below by the minimum of the spectral thresholds of one-dimensional operators $l(B)$ and $\widetilde{l}(B)$ and the two-dimensional  Schr\"{o}dinger operators with the potentials $\frac{1}{x^2+y^2} \left(\int_{\sqrt{x^2+y^2}}^{r_0}s B(s)\,\mathrm{d}s\right)^2$ and $\frac{1}{x^2+y^2} \left(\int_0^{\sqrt{x^2+y^2}}s B(s)\,\mathrm{d}s\right)^2$, which is not possible to obtain from a standard estimate (\ref{Berezin bound}). \end{remark}

\bigskip \section*{Acknowledgements}

The work of D.B. is supported by Czech Science Foundation (GACR), the project 14-02476S "Variations, geometry and physics'' and project SMO "Posileni mezinarodniho rozmeru vedeckych aktivit na Prirodovedecke fakulte OU v Ostrave" No. 0924/2016/SaS.
 F.T. is a member of  the ANR GeRaSic (G\'eom\'etrie spectrale, Graphes, Semiclassique).

D.B. appreciates the hospitality in Institut Fourier where the preliminary version of the paper was prepared.

\end{document}